\documentclass{amsart}
\usepackage{amssymb, latexsym}

\newtheorem{theorem}{Theorem}
\newtheorem{lemma}{Lemma}
\newtheorem{remark}{Remark}
\newtheorem{definition}{Definition}

\begin{document}
\title{Two Algorithms in Group Theory}
\author{Rita Gitik}
\address{ Department of Mathematics \\ University of Michigan \\ Ann Arbor, MI, 48109}
\email{ritagtk@umich.edu}
\date{\today}

\begin{abstract}
We present a new algorithm deciding if the intersection of a quasiconvex subgroup of a negatively curved group with a conjugate is finite. We also give a short proof of decidability of the membership problem for quasiconvex subgroups of finitely generated groups with decidable word problem.
\end{abstract}


\maketitle


\section{Introduction}

Let $H$ be a subgroup of a group $G$ and let $g$ be an element of $G$.
The subgroup $g^{-1} H g$ is called the conjugate of $H$ by $g$.
We would like to check if the intersection $H \cap g^{-1} H g$ is finite.
This is a very old problem, closely connected to the study of the behavior of different lifts
of subspaces of topological spaces in  covering spaces.

However in general, the question if a subgroup of a group is finite 
is undecidable. That was shown by Adian in 1957, \cite{Ad1}, and independently, by Rabin in 1958, \cite{Ra}.
The modern statement of Adian-Rabin theorem utilizes a notion of a Markov property for finitely presented groups.

\begin{definition}

A property $M$ of finitely presented groups is called Markov if it is preserved under group isomorphisms
and the following holds:

\begin{enumerate}
\item
There exists a finitely presented group $K$ with property $M$.

\item
There exists a finitely presented group $H$ which cannot be embedded in any finitely presented group with property $M$.
\end{enumerate}

\end{definition}

The Adian-Rabin theorem can be stated as follows.

\begin{theorem}
Let $M$ be a Markov property of finitely presented groups and let $G$ be a finitely presented group.
It is undecidable whether or not $G$ has property $M$.
\end{theorem}

Note that being a finite group is a Markov property. The group $K$ can be chosen to be the trivial group and the group $H$ can be chosen to be infinite cyclic group.

\bigskip

We would like to mention several families of subgroups which have well-understood intersections with their conjugates.

A subgroup $H$ is normal in $G$ if $H = g^{-1} H g$ for any $g \in G$. 
The study of normal subgroups goes back to the origins of group theory. 
The concept was introduced by \'{E}variste Galois at the beginning of the 19th century, 
who called  normal subgroups "invariant subgroups", \cite{Ga}.
G. Baumslag, Boone, and B. Newmann showed in 1959 that
being normal is an undecidable property for a subgroup, \cite{B-B-N}.

A subgroup $H$ is malnormal in $G$ if for any $g \in G$ such that $g \notin H$
the intersection $H \cap g^{-1} H g$ is trivial. This concept for infinite groups was introduced
by B. Baumslag in 1968, \cite{Ba}.
However, malnormality was investigated in finite groups at the end of the
19th century by Ferdinand Georg Frobenius, \cite{Fr}.
A proper nontrivial malnormal subgroup of a finite group is called  a Frobenius complement or a Frobenius subgroup.
Bridson and Wise showed in 2001 
that malnormality of a finitely generated subgroup in a negatively curved group is undecidable, \cite{B-W}, however
the author proved in 2016, \cite{Gi3}, that malnormality is decidable for a torsion-free quasiconvex subgroup of a 
negatively curved group.

Malnormality of a subgroup has been generalized in different ways. One of them, namely the height, introduced  by the author in 1995, \cite{G-M-R-S},
has been used by Agol in 2013 in his proof of  Thurston's conjecture that hyperbolic $3$-manifolds are
virtual bundles (over a circle with fiber a surface), \cite{Th}, \cite{Ag}, and \cite{A-G-M}.

A subgroup $H$ of $G$ is almost malnormal in $G$ if for any $g \in G$ such that $g \notin H$ the intersection $H \cap g^{-1}Hg$ is finite.
Using a result of Rips from 1982, \cite{Ri}, Bridson and Wise showed in 2001, \cite{B-W}, that almost malnormality of a finitely generated subgroup of a negatively curved group is undecidable. However, the author showed in 2016, \cite{Gi3}, that almost malnormality is decidable for quasiconvex subgroups of negatively curved groups. An informative paper on malnormality and almost malnormality was published by de la Harpe and Weber in 2014, \cite{H-W}.

Most subgroups are neither normal nor malnormal, so the study of the intersection pattern
of conjugates of a subgroup is a challenging problem.
We restrict ourselves to the special case of $H$ being quasiconvex and $G$ being negatively curved.

\section{Notation and Definitions}

Let $X$  be a set and let $X^* = \{x,x^{-1} |x \in X \}$, where for $x \in X$ we define $(x^{-1})^{-1} =x$.
Let $G$ be a group generated by the set $X$.
As usual, we identify a word in $X^*$ with the corresponding element in $G$. 

Let $Cayley(G)$ be the Cayley graph of $G$ with respect to the generating set $X$. 
The set of vertices of $Cayley(G)$ is $G$,  the set of edges of $Cayley(G)$ is $G \times X^*$, 
and the edge $(g,x)$ joins the vertex $g$ to $gx$. The Cayley graph was first considered by Cayley in 1878, \cite{Ca}.

A path  $p$ in $Cayley(G)$ is a sequence of edges of the following form:
$p=(g,x_1)(gx_1,x_2)  \cdots (gx_1x_2 \cdots x_{n-1},x_n)$. 
The length of the path $p$ is the number of edges forming it. 
A geodesic between two vertices in $Cayley(G)$ is a shortest path in $Cayley(G)$ connecting these vertices.

A group $G$ is $\delta$-negatively curved if any side of any geodesic triangle in $Cayley(G)$
belongs to the $\delta$-neighborhood of the union of the two other sides.
Negatively curved groups were introduced by Gromov in 1987, \cite{Gro}. Negatively curved groups are also called word hyperbolic
and Gromov hyperbolic groups.

Note that $Cayley(G)$ can be effectively constructed if and only if the word problem in $G$ is decidable. The word problem asks if there exists an algorithm to decide if any word in the alphabet $X^*$ represents the trivial element of $G= \langle X|R \rangle$. The word problem was introduced by Dehn in 1911, \cite{De}. It was shown  by Novikov in 1955, \cite{No}, and independently, by Boone in 1958, \cite{Bo},
that the word problem in groups in undecidable. However, it follows from the work of Greendlinger in 1960, \cite{Gre}, that the word problem is decidable in negatively curved groups.

A subgroup $H$ of $G$ is $K$-quasiconvex in $G$ if any geodesic in $Cayley(G)$ with the endpoints at $H$
belongs to the $K$-neighborhood of $H$. Quasiconvex subgroups were introduced by Gromov in 1987, \cite{Gro}.

\begin{remark}
G. Baumslag, C. F. Miller III, and Short showed in 1992 that given a finite presentation $<X|R>$ for a group $G$, it is undecidable 
if $G$ is negatively curved, \cite{B-M-S}. However, if $G$ is known to be negatively curved, a negative
curvature constant $\delta$ can be determined.
That was demonstrated by Epstein and Holt in 2000, \cite{E-H}, and independently, by Papasoglu in 1996, \cite{Pa}.

Bridson and Wise showed in 2001 that the property of being quasiconvex is undecidable for a finitely generated subgroup
of a torsion-free negatively curved group, \cite{B-W}. However, I. Kapovich showed  in 1996 that if a subgroup of a negatively curved group
is known to be quasiconvex, then a quasiconvexity constant $K$ for such a subgroup can be computed, \cite{Ka}.
\end{remark}

\section{An Algorithm Deciding If the Intersection of a Quasiconvex Subgroup of a Negatively Curved Group with a Conjugate Is Finite}

Input: a finite presentation $<X|R>$ for a negatively curved group $G$, a finite generating set for a quasiconvex subgroup $H$ of  $G$,
and an element $g \in G$ which is not in $H$.

Output: a finite group isomorphic to $H \cap g^{-1} H g$ or a statement that the intersection is infinite.

\begin{enumerate}
\item Find a constant $\delta$ (not necessarily minimal) of negative curvature of $G$.
This can be done using the results of  Epstein and Holt, \cite{E-H}, or of Papasoglu, \cite{Pa}.

\item
N. Brady in 2000, \cite{Br}, showed that the orders of finite subgroups of a $\delta$-negatively curved group $G$ generated by a finite set $X$ 
is bounded by a constant $C =(2|X|)^{2 \delta +1} +1$.  A similar result was obtained  independently by Bogopolskii and Gerasimov in 1996, \cite{B-G}. 

Make a list $L$ of all finite groups with fewer than $C$ elements.
This can be done, for example, by considering the multiplication tables of group elements.
Note that all finite groups are negatively curved because their Cayley graphs have finite diameters.

\item The author showed in 1996, \cite{Gi1}, that  conjugation, in general, does not preserve quasiconvexity. However,
the author proved in 1997, \cite{Gi3}, that a conjugate of a quasiconvex subgroup of a negatively curved group is quasiconvex. It follows that 
$g^{-1} H g$ is a quasiconvex subgroup of $G$. Gromov proved in 1987, \cite{Gro}, that the intersection of two quasiconvex subgroups of a negatively curved
group is quasiconvex. Hence, the subgroup $H \cap g^{-1} H g$ is quasiconvex in $G$. Bridson and Haefliger proved in 1999 that a quasiconvex subgroup of a negatively curved group is negatively curved, \cite{B-H}, p.462.  Therefore the group $H \cap g^{-1} H g$ is negatively curved.

The isomorphism problem in groups asks if there exists an algorithm to decide if any two presentations define isomorphic groups. The isomorphism problem was introduced by Dehn is 1911, \cite{De}. It was shown by Adian in 1957, \cite{Ad2}, and independently by Rabin in 1958, \cite{Ra}, that the isomorphism
problem in groups is undecidable. However, Dahmani and Guirardel showed in 2011 that the isomorphism problem for negatively curved groups is decidable, \cite{D-G}. Therefore  we can determine whether  $H \cap g^{-1} H g$ is finite by checking if it is isomorphic to an element of $L$.

If positive, output the intersection. 

If negative, output the statement that the intersection is infinite.
\end{enumerate}

\section{The Membership Problem}

We will need additional notation. Denote the equality of two words in $X^*$ by $\equiv$.
The length of the word $w$ is the number of symbols from $X^*$ forming $w$. Denote the length of the word $w$ by $|w|$.

The membership problem for a subgroup $H$ of $G= \langle X|R \rangle$ asks if there exists an algorithm which for any word $w$ in the alphabet $X^*$ decides whether or not $w$ represents an element of $H$. The membership problem is also called the generalised word problem. 
If $G$ has a decidable membership problem for the trivial subgroup, then $G$ has a decidable word problem. As the word problem, in general, is undecidable, \cite{No} and \cite{Bo}, the membership problem, in general, is undecidable.

The solution of the word problem in negatively curved groups follows from the result proved by Greendlinger in  1960, \cite{Gre} for certain small cancellations groups. A good exposition of this result was given by Lyndon and Schupp in 1977, \cite{L-S}, p.249. That result was generalized to negatively curved groups by Gromov in \cite{Gro}.  

\begin{theorem}
Let $G$ be a negatively curved group. There exists a finite  presentation $G=\langle X|R \rangle$, called  Dehn's presentation, with the following property. If $w$ is a non-trivial freely reduced word in $X^*$ such that $w=1_G$ then there exists a relator $r \in R$ and an initial subword $v$ of $r$ with $|v| > \frac{1}{2}|r|$ such that $v$ is a subword of $w$.
\end{theorem}

A good exposition of Dehn's presentation was given by Bridson and Haefliger in 1999, \cite{B-H}, p.450.

The solution of the word problem for a negatively curved group $G$ is given by the following procedure, called Dehn's algorithm.
A good exposition of Dehn's algorithm can be found in \cite{B-H}, p.449 and in \cite{L-S}, p.246.

Start by choosing Dehn's presentation $\langle X|R \rangle$ for $G$. Let $w$ be a non-trivial freely reduced word in $X^*$. 

\emph{For any relator $r \in R$ check if there exists an initial subword $v$ of $r$ with $|v| > \frac{1}{2}|r|$ such that $v$ is a subword of $w$. If no, then $w \neq 1_G$. If yes, let
$r \equiv vu$, replace the word $v$ in $w$ by the word $u^{-1}$ and freely reduce. Denote the resulting freely reduced word by $w_1$. As the words $v$ and $u^{-1}$ represent the same element in $G$, it follows that the words $w$ and $w_1$ represent the same element in $G$. As $|v| > |u|$, it follows that
$|w| > |w_1$. Repeat the procedure with the word $w_1$.}
 
Dehn's algorithm terminates in at most $|w|$ steps. If it terminates with the trivial word, then $w=1_G$.

\bigskip

It was shown by Rips in 1982, \cite{Ri}, that the membership problem is undecidable for arbitrary subgroups of  negatively curved groups. A good exposition of that result can be found in \cite{B-M-S}. However,
the membership problem for quasiconvex subgroups of negatively curved groups is decidable which was shown, for example, by the author in 1995, \cite{Gi1}, and independently, by Farb in 1994, \cite{Fa}, and by I. Kapovich in 1995, \cite{Ka}. The author in 2016, \cite{Gi3}, and independently, Kharlampovich, Miasnikov, and Weil in 2017, \cite{K-M-W}, gave new proofs of that fact.

We present a short solution of the membership problem for quasiconvex subgroups of finitely generated groups with decidable word problem. Our solution utilizes the concept of a weakly Nielsen generating set of a subgroup, introduced by the author in 1995, \cite{Gi1}. 

Nielsen generating sets are an important tool in the study of free groups. They originated in work of Nielsen in 1921, \cite{Ni}.
The main characteristic of Nielsen generating sets is a small amount of cancellations between all the members. A good description of Nielsen generating sets and their applications in free groups was given by Magnus, Karrass, and Solitar in 1966, \cite{M-K-S}, p.128.

\begin{definition}
Let $H$ be a subgroup of a group $G=\langle X|R \rangle$. We say that a finite generating set $S=\{s \equiv l_in_ir_i | 1 \leq i \leq m, n_i \neq 1 \}$
of $H$ is Nielsen if the cancellations in any freely reduced product of elements of $S$ do not affect the words $n_i$.
\end{definition} 

A theorem of Nielsen proven in 1921, \cite{Ni}, states that any finitely generated subgroup of a free group has a Nielsen generating set.

The existence of generating sets with similar strong noncancellation properties was a topic of extensive research which showed that Nielsen generating sets are very rare. See, for example, a paper of Collins and Zieschang from 1988, \cite{C-Z}. However a modified version of the Nielsen generating set turned out to be very useful. 

\begin{definition}
Let $H$ be a subgroup of the group $G=\langle X|R \rangle$. 
We say that a finite generating set $S$ of $H$ is weakly Nielsen
if for any $h \in H$ and for any shortest word $w$ in $X^*$ representing $h$ in $G$ there exist finitely many elements $s_i \in S$ and
decompositions $s_i \equiv l_in_ir_i$, (which depends on $w$,) with $n_i \neq 1$ 
such that $h=s_1 \cdots s_m$ and $w \equiv l_1n_1n_2 \cdots n_m r_m$.
\end{definition}

We need an additional definition.

\begin{definition} 
The label of the path  
$p=(g,x_1)(gx_1,x_2)  \cdots (gx_1x_2 \cdots x_{n-1},x_n)$ 
in $Cayley(G)$ is the word $Lab (p) \equiv x_1 \cdots x_n $. 

The inverse of a path $p$ is denoted $\bar{p}$.
\end{definition}

The author proved in 1996 that quasiconvex subgroups of finitely generated groups have weakly Nielsen generating sets, \cite{Gi1}.
Below is a new proof of that fact.

\begin{lemma}
Let $H$ be a $K$-quasiconvex subgroup of a  group $G$, generated by a finite set $X$.  
Then $H$ has a weakly Nielsen generating set consisting of elements of $G$ of length at most $2K+1$.
\end{lemma}
\begin{proof}
Consider $h \in H$. Let $w$ be a shortest word in $X^*$ representing $h$ in $G$, and let $\gamma$ be a geodesic in $Cayley(G)$ beginning at $1_G$
with $Lab(\gamma) \equiv w$. Let $v_0=1_G, v_1, \cdots, v_n=h$ be the vertices of $\gamma$ listed in order, and let $e_1, e_2, \cdots, e_n$ be the edges of $\gamma$ listed in order, so the length of $\gamma$ is $n$. 
As $H$ is $K$-quasiconvex in $G$, for any vertex $v_i$ of $\gamma$ there exists a path $t_i$ not longer than $K$, which joins $v_i$ to an element of $h$.
Note that some $t_i$ might be empty. Then 

$h = (Lab(\overline{t_0}) Lab (e_1) Lab(t_1)) (Lab(\overline{t_1}) Lab(e_2) Lab(t_2)) \cdots (Lab(\overline{t_n}) Lab(e_n) Lab(t_{n+1}))$. 

Let $l_i \equiv Lab(\overline{t_{i-1}}), n_i \equiv Lab(e_i), r_i \equiv Lab(t_i)$, and $s_i \equiv l_in_ir_i$.

Let $S$ be the union of all $s_i$, constructed for all $h \in H$. By construction, each $s_i$ is no longer than $2K+1$, hence as the set $X$ is finite,
the set $S$ is also finite. By construction, $S$ is a weakly Nielsen generating set of $H$.
\end{proof}

\begin{theorem}
Let $G$ be a group generated by a finite set $X$ and let $H$ be a $K$-quasiconvex subgroup of $G$. If $G$ has a decidable word problem 
then the membership problem for $H$ in $G$ is decidable.
\end{theorem}
\begin{proof}
Let $g$ be an element of $G$. As $G$ has a decidable word problem, we can find a shortest word $w$ in $X^*$ representing $g$ in $G$. Let the length
of $w$ be $n$. Lemma 1 states that $H$ has a weakly Nielsen generating set $S$.
If $g$ belongs to $H$ then, by definition of a weakly Nielsen generating set, $g$ can be written as a product of at most $n$ elements of $S$.
As $S$ is finite, we can generate all the products of at most $n$ elements of $S$. As $G$ has a decidable word problem, we can check if any of these products is equal to $g$.
\end{proof}

\section{Acknowledgment}

The author would like to thank Peter Scott for helpful conversations.

\end{document}